\documentclass[letterpaper,12pt,titlepage,oneside,final]{article}

\newcommand{\href}[1]{#1} %
\usepackage{ifthen}
\usepackage{amsmath}
\usepackage{amsthm}
\usepackage{tikz}
\usepackage{enumerate}

\newtheorem{definition}{Definition}

\newtheorem{lemma}{Lemma}
\newtheorem{theorem}{Theorem}

\newtheorem{prop}{Proposition}
\newtheorem{corollary}{Corollary}
\newboolean{PrintVersion}
\setboolean{PrintVersion}{false}
\usepackage{amsmath,amssymb,amstext} %

\usepackage[pdftex,pagebackref=false]{hyperref} %

\usepackage[backend=bibtex]{biblatex}
\hypersetup{
    plainpages=false,       %
    unicode=false,          %
    pdftoolbar=true,        %
    pdfmenubar=true,        %
    pdffitwindow=false,     %
    pdfstartview={FitH},    %
    pdfnewwindow=true,      %
    colorlinks=true,        %
    linkcolor=blue,         %
    citecolor=green,        %
    filecolor=magenta,      %
    urlcolor=cyan           %
}
\ifthenelse{\boolean{PrintVersion}}{   %
\hypersetup{	%
    citecolor=black,%
    filecolor=black,%
    linkcolor=black,%
    urlcolor=black}
}{} %

\usepackage[automake,toc,abbreviations]{} %
\let\origdoublepage\cleardoublepage
\newcommand{\clearemptydoublepage}{%
  \clearpage{\pagestyle{empty}\origdoublepage}}
\let\cleardoublepage\clearemptydoublepage

\newcommand{\C}{{\mathbb C}}
\newcommand{\F}{{\mathbb F}}

\newcommand{\Q}{{\mathbb Q}}

\newcommand{\R}{{\mathbb R}}
\newcommand{\Z}{{\mathbb Z}}
\newcommand{\M}{{\mathcal M}}

\renewcommand{\H}{{\mathcal H}}
\renewcommand{\L}{{\mathcal L}}

\DeclareMathOperator{\N}{N}

\DeclareMathOperator{\bN}{\mathbb{N}}

\newcommand{\A}{{\mathbb A}}
\newcommand{\p}{{\mathbb{P}^1(\mathbb{C})}}

\newcommand{\bZ}{{\mathbb Z}}

\newcommand{\bC}{{\mathbb C}}

\newcommand{\bA}{{\mathbb A}}
\newcommand{\bQ}{{\mathbb Q}}

\begin{document}

\pagestyle{empty}
\pagenumbering{roman}

\begin{titlepage}
        \begin{center}
        \vspace*{1.0cm}

        \Huge
        {\bf  A Look at Chowla's Problem}

        \vspace*{1.0cm}

        \Large
        Andean Medjedovic \\

	\vspace{10mm}

	\normalsize
\begin{center}\textbf{Abstract}\end{center}

In this paper we look at the history behind Chowla's problem on the solutions to $L(1,f) = 0$ for periodic $f$. We focus on the results given by Baker, Birch and Wirsing on the topic. We briefly discuss recent results due to Chatterjee, Murty and Pathak which give a full solution
when combined with the work of Baker, Birch and Wirsing.

        \end{center}
\end{titlepage}

\pagestyle{plain}
\setcounter{page}{2}

\cleardoublepage

\renewcommand\contentsname{Table of Contents}
\tableofcontents
\cleardoublepage
\phantomsection    %

\pagenumbering{arabic}

\section{Dirichlet's Theorem}
In 1831, Peter Gustav Lejeune Dirichlet introduced his namesake character in order to prove that the arithmetic progression $an+b$,
for natural numbers $a, b$ and $n$ with $\gcd (a,b) = 1$, takes on prime values infinitely often. The associated $L$-series of the characters were
instrumental in his study. During his exploration, Dirichlet managed to demonstrate the value of such $L$-series at $1$ is non-zero. This result later proved to be crucial
to Chowla's investigation. We begin with a quick review of the relevant details.

\begin{definition}
	We say a function $\chi:\bZ \to \bC$ is a Dirichlet character modulo $k$ if:
	    \begin{itemize}
				\item $\chi$ is completely multiplicative, $\chi(nm) = \chi(n)\chi(m)$ for any $n,m \in \bN$
				\item $\chi$ is $k$-periodic, $\chi(n+k) = \chi(n)$ for any $n\in \bN$
				\item $\chi(n) \ne 0$ for $(n,k) = 1$ and $\chi(n) = 0$ for $(n,k) \ne 1$
		\end{itemize}
\end{definition}
	Furthermore, a character is principal if $\chi(n) \in \{0,1\}$ for all $n \in \bN$.
	Notice that if $l$ is a multiple of $k$, a character on $k$ is also a character on $l$. We say a character (modulo $l$) is primitive if
	it is not a character of modulo $k$ for $k < l$. We use $X(k)$ to denote the set of Dirichlet characters modulo $k$.

\begin{prop}
The Dirichlet characters satisfy an orthogonality relation. If $\chi$ is a non-principal character modulo $k$ then
	$$\sum_{a \in \bZ/k\bZ}\chi(a) = 0$$
	$$\sum_{\chi \in X(k)} \chi(a) =0$$
Where $a\ne 1$ in the second sum. If $a = 1$ then the second sum is $\phi(a)$.
\end{prop}

\begin{definition}
	Given a character we define the $L$-series of that character to be the analytic continuation of the sum:
		$$ L(s,\chi) = \sum_{n=1}^{\infty} \frac{\chi(n)}{n} $$
	where $s \in \bC$. One can check as an exercise the any non-principal $L$-series is holomorphic for real part of $1> s > 0$.
\end{definition}

We now prove the aforementioned theorem of Dirichlet\cite{dirichletthm}.
\begin{theorem}
	If $\chi$ is a non-principal character then:
	$$L(1,\chi) \ne 0$$
\end{theorem}
\begin{proof}
	Let $X(k)$ be the set of all Dirichlet characters of modulus $k$ and consider the product
	$$\zeta_k(s) = \prod_{\chi \in X(k)} L(s,\chi)$$
	as $s \to 1^+$. We first show that $\zeta_k(s)$ has a singularity at $s = 1$.
	Re-writing the above as a Euler product we see that,

	$$\zeta_k(s) = \prod_{p\nmid k}\frac{1}{(1- \frac{1}{p^{|p|s}})^{\frac{\phi(k)}{|p|}}}$$
	where $|p|$ is the order of $p$ in $\bZ/k\bZ$. We now invoke a theorem of Landau stating that if the coefficients of a $L$-series are at least zero, the real point
	on the abscissa is a singularity of the $L$-series. Suppose $\zeta_k(s)$ does not have a singularity at $s=1$. Since the $L$-series of each character is holomorphic at $s=1$,
	the abscissa must be at some $\sigma \leq 0$. But this is ridiculous, as we can see that for  $s =\frac{1}{ \phi(k)}$ the product is unbounded, we have

	$$\zeta_k(s) = \prod_{p\nmid k}\frac{1}{(1- \frac{1}{p^{|p|s}})^{\frac{\phi(k)}{|p|}}} \geq \prod_{p\nmid k} \sum_{i= 0}^{\infty} \frac{1}{p^{\phi(k)si}}
	\geq \sum_{p \nmid k} \frac{1}{p}$$

	And Euler showed that the sum of the reciprocal of primes diverges.

	So $\zeta_k(s)$ indeed has a singularity. If $\chi$ is a principal character then $L(\chi,s)$ has a simple pole at $s=1$, If $L(\chi,s) = 0$ for any other (non-principal) character
	then it follows that $\zeta_k(s)$ would not have a pole at $s=1$. A contradiction, thus $L(\chi,s) \ne 0$.
\end{proof}
\begin{definition}{Cyclotomic Polynomials}

	We say the minimal polynomial over $\bQ$ of any primitive root of $n$, say $e^{2\pi i /n}$, is the $n^{th}$ cyclotomic polynomial, $\Phi_n$.
	One can check that it is also possible to define it with the following:
	$$\Phi_n(x) = \prod_{\substack{1 \leq k \leq n \\ (k,n) = 1}} (x - e^{2\pi i k/n})$$
\end{definition}

\section{Chowla's Question}
Sarvadaman Chowla asked the following natural generalization of Dirichlet's theorem to Carl Ludwig Siegel in 1949\cite{chowla70}. Let $f$ be a non-zero $q$-periodic function from $\bZ$ to $\bQ$ that
takes $0$ for values not relatively prime to $q$ (that is $f(p) = 0$ if $(p,q) > 1$). Is it possible that
$$L(1,f) = \sum_{n=1}^{\infty}\frac{f(n)}{n} = 0?$$

\begin{lemma}{\label{0sum}}
	If such an $f$ exists we must have $\sum_{n =1}^{q}f(n) = 0$.
\end{lemma}
\begin{proof}
	Suppose, without loss of generality that $\sum_{n =1}^{q}f(n) = p > 0$. Now consider the sum we are interested up to some large $qy$, $y\in \bN$ and sum by parts,
	$$\sum^{qy}\frac{f(n)}{n} = \sum^{qy}f(n) + \int_1^{qy}\frac{\left(\sum^{t}f(n)\right)}{t^2} dt = py + \int_1^{qy} \frac{\frac{t}{q}p +O(1)}{t^2}dt = py + \ln(py) +O(1)$$
	So as $y$ increases the sum is unbounded.
\end{proof}

 Chowla was able to prove his conjecture in the case that $f$ was an odd periodic function with values in $\{-1,1,0\}$ and with prime period, $q$, with $\frac{1}{2}(q-1)$ also prime.
 In a letter to Chowla, Siegel extended the result to include all rational-valued odd periodic functions with prime period, $q$.

\section{Baker, Birch, and Wirsing}

\begin{theorem}[Baker, Birch, and Wirsing, 1973\cite{baker73}]{\label{thm1}}

	Suppose $f : \bZ \to \bA$ is non-zero, (where $\bA$ are the algebraic numbers) with period $q$ so that:
	\begin{enumerate}[(i)]
			\item $f(r) = 0$ if $1 < (r,q) < q$
			\item $\Phi_q$ is irreducible over $\bQ(f(1), \cdots, f(q))$
	\end{enumerate}
	Then
	\begin{equation}{\label{1}}
	L(1,f) = \sum_{n=1}^{\infty} \frac{f(n)}{n} = 0.
	\end{equation}
	cannot hold.
\end{theorem}

Baker, Birch and Wirsing published this paper in $1973$, a bit more than $20$ years after Chowla posed the problem. While the authors were working on the problem, Chowla expanded on his previous
solution, proving there is no $f$ so that $\ref{1}$ holds even if we allow $f$ to be even so long as we force $f(0) = 0$. As one can see for themselves, the theorem presented by
Baker et al. is more general still.

We mention the necessity of both conditions $(i)$ and $(ii)$. Suppose we omitted condition $(i)$, and consider the function, $f$ defined by

$$L(s,f) = (1-p^{1-s})^2\zeta(s)$$
where $\zeta(s)$ is, of course, the Riemann zeta funcation. As $s$ goes to $1$, $(1-p^{1-s})^2\zeta(s)$ goes to $0$ since $\zeta(s)$ has a only a simple pole.
It remains to show that $f(n)$ is periodic. I claim the period is $p^2$. Indeed expand each term of the RHS:
$$(1-p^{1-s})^2\frac{1}{n^s} = \frac{1}{n^s} - \frac{2p}{(np)^s} + \frac{p^2}{(np^2)^s}$$
If we define $(a|b)$ to be $1$ when $a|b$ and $0$ otherwise, for integers $a,b$ then we can write out $f(n)$ explicitly like so:
$$f(n) = \frac{1 - 2p(p|n) +p^2(p^2|n)}{n^2}.$$
So $f(n+p^2) = f(n)$ and so condition $(i)$ is necessary.

Condition $(ii)$ is also necessary, it is possible for $L(1,f)$ to be algebraically dependent for certain $f$. One such example is for the $2$ dirichlet characters modulo $12$ with values given
below.

\begin{center}

\begin{tabular}{c c c c c }
\hline
$n$ & $1$ & $5$ & $7$ & $11$ \\
\hline
$\chi_1(n)$ & $1$ & $-1$ & $1$ & $-1$ \\
$\chi_2(n)$ & $1$ & $-1$ & $-1$ & $1$ \\
\hline
\end{tabular}\\
\end{center}
Then $L(1,\chi_1) = \frac{\pi}{3}, L(1,\chi_2) = \frac{\pi}{2\sqrt{3}}$ and so $2\chi_2 -
\sqrt{3}\chi_1 = f$ implies that $L(1,f) = 0$.
The $12^{th}$ cyclotomic is $x^4 - x^2 +1$ which is reducible over $\bQ(\sqrt{3})$, so condition $(ii)$ is necessary.

Immediately from the theorem we get a non-trivial corollary:
\begin{corollary}
	Suppose $(q,\phi(q)) = 1$ and let $X'(q)$ be the set of non-principal characters modulo $q$. Then $L(1,\chi)$ are linearly independent over $\bQ$ for $\chi\in X'(k)$.
\end{corollary}
\begin{proof}
	Suppose we have some relation $\sum_{i=1}^{N}\alpha_i L(1,\chi_i) = 0$ where the $\alpha_i \in \bQ$. Then we can define $f = \sum_{i=1}^{N}\alpha_i \chi_i$ and get a contradiction
	to the above theorem, if our $f$ satisfies condition $(ii)$ (it clearly satisfies $(i)$). So this cannot happen. We get the corollary by an implication of condition $(ii)$.
	$\Phi_q$ is irreducible over $\bQ_{\phi(q)}$ if $[ \bQ_{\phi(q)}(e^{2\pi i /q}): \bQ_{\phi(q)}] = \phi(q)$. Computing the degree of the LHS gives
	$$\frac{\phi(q)}{\phi((\phi(q),q))} = \phi(q)$$
	Thus $(\phi(q) ,q)$ is one or two. So $f$ will satisfy $(ii)$ if $(\phi(q),q) = 1$ and then the corollary holds.
\end{proof}

\section{Preliminary Results}

We work towards a proof of {\ref{thm1}} following the argument given by Baker, Birch and Wirsing. The general idea of the
proof is to transform our sum using shifted versions of it to prove that it must be $0$ on a orthogonal set. This would imply that
$f= 0$ in the first place. To begin, let $F_q$ be the set of $f:\bZ \to \bA$ with period $q$ so that
$$L(1,f) = \sum_{n=1}^{\infty}\frac{f(n)}{n} = 0$$
and let $G_q$ be the set of functions of form
$$g(s) = \frac{1}{q}\sum_{r=1}^{q} f(r)e^{-2\pi r s /q}$$
for an $f \in F_q$. Note that \ref{0sum} implies $g(0) = 0$ for any $g \in G_q$.

\begin{lemma}{\label{log}}
	$g\in G_q$ if and only if $$\sum_{s=1}^{q-1}g(s)\log(1-e^{2\pi i s /q}) = 0$$
\end{lemma}
\begin{proof}
	First we show that $f(r) = \sum_{s = 1}^{q}g(s)e^{2\pi i r s/q}$. Indeed, substitute $g(s)$ in and sum over $s$:
	\begin{align*}
	\sum_{s = 1}^{q}g(s)e^{2\pi i r s/q} 	&= \frac{1}{q}\sum_{s=1}^{q}\sum_{l=1}^{q}f(l) e^{2\pi i s(l-r)/q} \\
						&= \frac{1}{q}\sum_{s = 1}^{q}f(r) + \frac{1}{q}\sum_{\substack{l \ne r \\ 1\leq l \leq q}} f(l)\frac{e^{2\pi i(r-l)} - 1}{e^{2\pi i(r-l)/q} -1} \\
						&= f(r) + 0  = f(r)
    \end{align*}
	Now to prove the actual lemma expand the taylor series of the logarithm:
	$$\sum_{s=1}^{q-1}g(s)og(1-e^{2\pi i s /q}) = -\sum_{k=1}^{\infty}\sum_{s=1}^{q-1}g(s) \frac{e^{2\pi i sk/q}}{k} = -\sum_{k =1}^{\infty}\frac{f(k)}{k} = 0$$
	Where the last equality comes from the definition of $f$.
\end{proof}

Here we invoke Baker's theorem on linear forms in logarithms. It turns out we don't need the actual bound given by Baker in \cite{baker68}, we just need to show a certain
linear form in logs is zero and so the coefficients in the form must be $0$, by Baker.

\begin{lemma}{\label{aut}}
	Let $\sigma$ be an automorphism of $\bA$ and $g \in G_q$. Then $\sigma g \in G_q$.
\end{lemma}
\begin{proof}
	Consider the set $\log(1 - e^{2\pi i s /q})$ for $s = 1, \cdots, q-1$ over $\bQ$. Let $\log(\alpha_i)$ be a linearly independent subset of maximum cardinality $i = 1, \cdots, t$ so
	they form a basis and:
	$$\log(1-e^{2\pi i s/q}) = \sum_{i=1}^{t}a_{rs}\log(\alpha_r)$$
	for some $a_{rs} \in \bQ$.
	Let $\beta_r = \sum_{s = 1}^{q-1}g(s) a_{rs}$.\\

	Then $\sum_{s=1}^{q-1}g(s)\log(1-e^{2\pi i s/q}) = 0$ implies that $$\sum_{i =1}^{t}\beta_i \log(\alpha_i) = \sum_{s=1}^{q-1}g(s)\sum_{r=1}^{t} a_{rs}\log(\alpha_r) = 0$$ since
	$$\sum_{s=1}^{q-1}g(s)\sum_{r=1}^{t} a_{rs}\log(\alpha_r) =\sum_{s=1}^{q-1}g(s)\log(1-e^{2\pi i s/q}).$$

	Then we have a linear form in logs with a sum of $0$,
	$$\sum_{i =1}^{t}\beta_i \log(\alpha_i) = 0.$$
	Since all of the $\log(\alpha_i)$ are linearly independent over $\bQ$ it follows that each $\beta_i$ is $0$. So $\sum_{s=1}^{q-1} g(s)a_{rs} = 0$. Now for any automorphism $\sigma$
	note that
	$$\sum_{s=1}^{q-1}\sigma(g(s)) \log(1- e^{2\pi i s/q}) = \sum_{r=1}^{t}\sigma(\sum_{s=1}^{q-1} g(s) a_{rs}) \log(\alpha_r) = 0$$
	By linearity of $\sigma$. The last sum is $0$ since each term in the sum over $r$ is $0$.
	By lemma $\ref{log}$, $\sigma g \in G_q$.
\end{proof}

\begin{lemma}{\label{hn}}
	Given an automorphism of $\bA$, $\sigma$, define $h$ to be the unique integer modulo $q$ by $\sigma^{-1}e^{2\pi i /q} = e^{2\pi i h /q}$. Define $f'(n) = \sigma f(hn)$
	for all $n$, then $f' \in F_q$.
\end{lemma}
\begin{proof}
	Since $f(r)$ and $g(r)$ are related by the formula in $\ref{log}$ then there is some $g(s) \in G_q$ so that
	\begin{align*}
	\sigma f(hn) &= \sigma \left( \sum_{s=1}^{q-1}g(s) e^{2\pi i hn s/q} \right) \\
			&= \sum_{s=1}^{q-1}\sigma (g(s)) \sigma(e^{2\pi i nhs/q}) \\
			&= \sum_{s=1}^{q-1}\sigma g(s) e^{2\pi i n s/q}
	\end{align*}

	Thus $\sigma f(hn) \in F_q$ if and only if $\sigma g(s) \in G_q$ by the way $G_q$ and $F_q$ were constructed. Also by $\ref{aut}$,
	$\sigma q(s) \in G_q$ and so $\sigma f(hn) \in F_q$.
\end{proof}

\section{The Theorem of Baker, Birch and Wirsing}
\begin{proof}
We are ready to prove the main result (Theorem {\ref{thm1}}). The goal is to show that if the properties $(i)$ and $(ii)$ are forced on $f$, $f = 0$.
Recall that $(ii)$ state that the minimal polynomial of $e^{2\pi i /q}$ is irreducible over $\bQ$ adjoined with the values that $f$ takes on.
Accordingly, we may choose some automorphism $\sigma$ that fixes $f(n)$ for all $n$ and takes $e^{2\pi i h/q}$ to $e^{2\pi i /q}$ for any $h$ with $(h,q) = 1$.
Then $\sigma f(hn) = f(hn)$ is in $F_q$.

Consider now the sum given by

	$$\sum_{\substack{h =1 \\ (h,q) =1}}^q\sum_{n=1}^{\infty} \frac{f(hn)}{n} = 0.$$
It is $0$ since each term summing over $h$ is $0$. Interchange the sums and group terms with the same denominator to get

	$$\sum_{n=1}^\infty \frac{\sum_{\substack{h =1 \\ (h,q) =1}}^q f(hn)}{n} = 0.$$

	The sum $\sum_{\substack{h =1 \\ (h,q) =1}}^q f(hn)$ is $0$ if $1 < (n,q) < q$ by $(i)$.
	If $(n,q) = q$ then the sum is $\phi(q)f(0)$ as each term is $f(0)$. Finally, if $(n,q) = 1$, then $hn$ takes on all values, $p$, modulo $q$, with $(p,q) = 1$ since we always have
	$(h,q) = 1$. Thus the sum in question is $\sum_{s = 0}^{q} f(s) - f(0) = -f(0)$ by $\ref{0sum}$.

	The idea now is to sum over the individual periods of $f$. Over any period of length $q$,say,  $(n+1, \cdots n+q)$ we have one multiple of $q$ and $\phi(q)$ integers coprime
	to $q$.

    \begin{align*}
	0 &= \sum_{n=1}^{\infty} \frac{\sum_{\substack{h =1 \\ (h,q) =1}}^q f(hn)}{n} \\
	    &= \sum_{m = 0}^{\infty} \sum_{s=1}^{q} \frac{\sum_{\substack{h =1 \\ (h,q) =1}}^q f(h(mq+s))}{mq+s} \\
		&= f(0)\sum_{m = 0}^{\infty} \left( \frac{\phi(q)}{q+mq} - \sum_{\substack{1\leq s\leq q \\ (s,q) = 1}} \frac{1}{s+mq}\right)
	\end{align*}
	Now each term in the sum over $m$ is less than than $0$, $\frac{\phi(q))}{q+mq} - \sum_{\substack{1\leq s\leq q \\ (s,q) = 1}} \frac{1}{s+mq} < 0$. This means that $f(0) = 0$.
	We want to show that the weighted sum $\sum_{n =1}^{q} \chi(n)f(n) = 0$ for arbitrary $\chi \in X(q)$. By orthogonality, that would mean $f(n) = 0$ for all $n$, as each sum can be
	viewed as a component of the multiplication of the orthogonal matrix $M = \chi_{i}(j)$ with the vector $[f(i)]_{(i,q) = 1}$ for some ordering of the $\phi(q)$ many $\chi \in X(q)$.
	By invertability of the matrix the result follows.

	If $\chi \in X(q)$ is principal then the result is true by the lemma $\ref{0sum}$. Suppose $\chi$ isn't principal an define
	$$b(n) = \sum_{h = 1}^{q}\chi(h) f(hn)$$
	so that
	$$b(1)\overline{\chi(n)} = b(n).$$
	To see this note that if $hn = m$ then $\chi(m)\overline{\chi(n)} = \chi(h)$ by complete multiplicativity. So then $\chi(m)\overline{\chi(n)}f(m) = \chi(h)f(hn)$ and we can sum over
	all the numbers coprime to $q$ from $1, \cdots, q$ (or all numbers from $1,\cdots, q$ since $f(m)$ is $0$ when $(m,q) > 1$)  which yields
	$$b(1)\overline{\chi(n)} = \sum_{m = 1}^{q}\chi(m)\overline{\chi(n)}f(m) = \sum_{h=1}^{q} \chi(h)f(hn) = b(n).$$
	But now

	$$b(1)L(1, \overline{\chi}) = \sum_{n=1}^\infty \frac{b(n)}{n} = \sum_{h =1}^{q}\chi(h)\sum_{n=1}^\infty\frac{f(hn)}{n} = 0$$
	and we have seen that $L(1,\overline{\chi})$ is not $0$ so $b(1) = \sum_{h = 1}^{q}\chi(h)f(h) = 0$. Since $\chi$ was arbitrary, $f = 0$.
	Which was what was wanted.

\end{proof}
\section{Complete Solution to Chowla's Problem}

We briefly review the other theorems in \cite{baker73} and more recent work done by Chatterjee, Murty, and Pathak \cite{murty18}.
In their paper on the problem Baker et al. were able to drop conditions $(i)$ and $(ii)$ and classify the solutions to $\ref{1}$ under a much weaker condition.
Namely, that $f$ must be odd.

\begin{theorem}{\label{2}}
	Let $q \geq 3$ be a natural number. Suppose $f$ is odd, algebraic valued and periodic modulo $q$ so that $\ref{1}$ holds. Then $f$ is contained in
	the span (over $\bA$) of the following depending on if $q$ is odd or even, respectively:
	\begin{equation}{\label{span}}
		f_l(n) = \left(\frac{sin(n\pi/q)}{sin(\pi/q)}\right)^l \hspace{5mm} l = 3, 5 \cdots, q-2
	\end{equation}
	\begin{equation}
		f_l(n) = \frac{cos(n\pi/q)}{cos(\pi/q)}\left(\frac{sin(n\pi/q)}{sin(\pi /q)}\right)^l \hspace{5mm} l = 3, 5, \cdot,s q-3
	\end{equation}
\end{theorem}

And moreover, if we insist that $L(1,f) = 0$ for a $q$-periodic function $f$ and condition $(i)$ holds then it was shown that $f$ has to be odd in the same paper. About a year ago, Chatterjee, Murty and Pathak were able to generalize the above result in their paper with the following theorem.

\begin{theorem}[Chatterjee, Murty, Pathak\cite{murty18}]
    Let $f$ be algebraically valued with period $q$. Let
    $$f_o = \frac{f(x) - f(-x)}{2} \hspace{1cm} f_e = \frac{f(x) + f(-x)}{2}$$
    be the odd and even parts of $f$, respectively. Then
    $$L(1,f) = 0 \iff L(1,f_o) = 0 \hspace{2mm}{\rm and}\hspace{2mm} L(1,f_e) = 0$$
\end{theorem}

Chatterjee et al. then provide a complete characterization of the even $f$ so that $L(1,f)=0$.

\begin{theorem}
    Consider the set of all $f$, where $f$ is an even algebraically valued $q$-periodic function so that $L(1,f) = 0$. This set is exactly the space spanned (over the algebraics) by
    $$\hat{F_{d,c}}$$
    Where $d |q$, $1 \leq c \leq d-1$ and $\hat{F_{d,c}}$ is the Fourier transform of
    $F_{d,c} = F^{(1)}_{d,c} - F^{(2)_{d,c}}$. Here $F^{(1)}_{d,c}$ and $F^{(2)}_{d,c}$
    are defined by
    \begin{equation*}
        F^{(1)}_{d,c}(n) = \begin{cases}
                                   1 & \text{if $n\equiv c$ modulo $q$} \\
                                   0 & \text{otherwise}
                            \end{cases}
    \end{equation*}
    \begin{equation*}
        F^{(2)}_{d,c}(n) = \begin{cases}
                                   1 & \text{if $n\equiv \frac{qc}{d}$ modulo $q$} \\
                                   0 & \text{otherwise}
                            \end{cases}
    \end{equation*}
\end{theorem}

Together with {\ref{2}}, this gives a complete description of the periodic functions with
$L(1,f) = 0$, solving Chowla's problem in full generality.

\section{Acknowledgments}

The author would like to thank Cameron Stewart for suggesting the topic and relevant results. We would also like to thank him
for teaching a course on Linear forms in Logarithms and exposing us to Baker's theorem. Finally, we would like to thank the audience to which this summary was presented, for pointing out
several clarifications.

\newpage

\printbibliography

\end{document}